\documentclass[12pt]{amsart}
\usepackage{amsmath}
\usepackage[latin1]{inputenc}
\usepackage{amssymb, amsmath}

\theoremstyle{plain}
\newtheorem{thm}{Theorem}[section]

\newtheorem{fact}[thm]{Fact}
\newtheorem{lem}[thm]{Lemma}
\newtheorem{cor}[thm]{Corollary}

\theoremstyle{definition}
\newtheorem{defi}[thm]{Definition}
\newtheorem{remark}[thm]{Remark}
\newtheorem*{question}{Question}

\def\Ind{\setbox0=\hbox{$x$}\kern\wd0\hbox to 0pt{\hss$\mid$\hss}
\lower.9\ht0\hbox to 0pt{\hss$\smile$\hss}\kern\wd0}
\def\Notind{\setbox0=\hbox{$x$}\kern\wd0\hbox to 0pt{\mathchardef
\nn=12854\hss$\nn$\kern1.4\wd0\hss}\hbox to
0pt{\hss$\mid$\hss}\lower.9\ht0 \hbox to
0pt{\hss$\smile$\hss}\kern\wd0}

\def\acl{\mathrm{acl}}

\title{On the Canonical Base Property}
\date{May 25, 2012}
\author{Ehud Hrushovski, Daniel Palac\'\i n, and Anand Pillay}
\thanks{The research leading to these results has received funding for the first author from the European Research Council under the European Union's Seventh Framework Programme (FP7/2007-2013) / ERC Grant agreement no. 291111.
The second author was partially supported by research project MTM 2011-26840 of the Spanish government and research project 2009SGR 00187 of the Catalan government. The third author was supported by EPSRC grant EP/I002294/1 and also by the Max Planck Institute in Bonn. }

\address{Hebrew University of Jerusalem}
\email{ehud@math.huji.ac.il}
\address{Universitat de Barcelona; Departament de L\`ogica, Hist\`oria i Filosofia de la Ci\`encia, Montalegre 6, 08001 Barcelona, Spain}
\email{dpalacin@ub.edu}
\address{University of Leeds}
\email{A.Pillay@leeds.ac.uk}

\begin{document}

\begin{abstract}  We give an example of a finite rank, in fact $\aleph_{1}$-categorical, theory where the canonical base property ($CBP$) fails. In fact
we give a ``group-like" example in a sense that we will describe below. We also prove, in a finite Morley rank context, that if all definable Galois groups are ``rigid" then $T$ has the $CBP$.
\end{abstract}

\maketitle

\section{Introduction and preliminaries}
The {\em canonical base property} or $CBP$ is a property of finite rank theories, the formulation of which was motivated by results of Campana in bimeromorphic geometry  and analogous results by Pillay and Ziegler in differential and difference algebraic geometry in characteristic $0$. The notion has been studied by Chatzidakis \cite{zoe}, Moosa and Pillay \cite{moo-pil} (where the expression $CBP$ was introduced)  and in a somewhat more general framework  by Palac\'\i n  and Wagner \cite{pal-wag}. The notion makes sense for arbitrary supersimple theories of finite $SU$-rank. But to avoid unnecessary abstraction we will restrict ourselves here to {\em stable theories $T$ of finite Morley rank} (and even more). We now state the $CBP$ giving definitions of some of the ingredients in an appropriate context  later in this section.

\begin{defi} $T$ has the $CBP$ if (working in $M^{eq}$ for a saturated $M\models T$, and over any small set of parameters), for any $a,b$ such that $tp(a/b)$ is stationary and $b$ is the canonical base of $tp(a/b)$, $tp(b/a)$ is {\em  semiminimal}, namely almost internal to the family of $U$-rank $1$ types.
\end{defi}

Here the canonical base of a stationary type $p(x) = tp(a/B)$, written $Cb(p)$, is the smallest definably closed subset $B_{0}$ of $dcl(B)$ such that $p$ does not fork over $B_{0}$ and $p|B_{0}$ is stationary. Given our totally transcendental hypothesis on $T$, $B_{0}$ will be the definably closure of a finite subtuple $b_{0}$ and we write $b_{0}$ for $Cb(p)$. 

\begin{remark} Chatzidakis shows in \cite{zoe} that with notation as above $tp(b/a)$ is always {\em analyzable} in the family of nonmodular $U$-rank $1$ types (which we know to be of Morley rank $1$). Hence the $CBP$ is equivalent to saying that $tp(b/a)$ is almost internal to the family of  nonmodular strongly minimal definable sets.
\end{remark} 

When $T$ is the theory of the many-sorted structure $CCM$ of compact complex manifolds (with predicates for analytic subvarieties of products of sorts), Pillay \cite{pil02} noted that results of Campana yield that the $CBP$ holds in the strong sense that $tp(b/a)$ is internal to the sort of the projective line over $\mathbb C$. This gives another proof that this sort is the only nonmodular strongly minimal set in $CCM$ up to nonorthogonality. Then Pillay and Ziegler \cite{PZ03} proved the analogous result where $T$ is the many sorted theory of definable (over some fixed set of parameters) sets of finite Morley rank in differentially closed field $(K,+,\cdot, \partial)$ with all induced structure:  $tp(b/a)$ is internal to the {\em constants}. This again gives another proof that the constants is the only nonmodular strongly minimal definable set in $DCF_{0}$ up to nonorthogonality. 

The following consequence of the $CBP$ for definable groups was observed in \cite{PZ03} for example, but we repeat the proof for the convenince of the reader.
\begin{fact} Assume $T$ has the $CBP$. Let $G$ be a definable group, let $a\in G$ and suppose $p(x) = stp(a/A)$ has finite stabilizer. Then $p$ is  semiminimal (in fact as above almost internal to the family on nonmodular strongly minimal sets). 
\end{fact}
\begin{proof} Assume $A = \emptyset$. Let $c\in G$ be generic in $G$ over $a$. As $Stab(p)$ is finite, by Lemma 2.6 of \cite{pil02}, $a$ is interalgebraic with $d = Cb(stp(c/ac))$. By the $CBP$, $stp(d/c)$ is  semiminimal, hence $stp(a/c)$ is semiminimal. As $a$ is independent from $c$ over $\emptyset$, $p$ is semiminimal.

\end{proof}

As was pointed out in \cite{PZ03} the $DCF_{0}$ case of Fact 1.3 yields an account of Mordell-Lang for function fields in characteristic $0$, following the lines of Hrushovski's proof \cite{udi96} but  with considerable simplifications.
With suitable definitions the $CBP$ holds for the category of finite $SU$-rank sets in $ACFA_{0}$ and the ``group version" likewise yields a quick account of Manin-Mumford.\\

We should also remark that Chatzidakis deduces some on the face of it stronger statements from the $CBP$. For example, suppose $b = Cb(tp(a/b)$, then $tp(b/acl(a)\cap acl(b))$ is  semiminimal (and again almost internal to the family of nonmodular strongly minimal sets). And in fact always (not necessarily assuming the $CBP$), in this context $tp(b/acl(a)\cap acl(b))$ is analysable in the nonmodular strongly minimal sets.  But we will not be using the latter. Note also that $T$ has the $CBP$ if and only it is has the $CBP$ after adding constants. \\

Note also that the canonical base property can be seen as a generalization of $1$-basedness: $T$ is $1$-based if whenever $b = Cb(tp(a/b))$ then $b\in acl(a)$. This point of view is profitably pursued in \cite{pal-wag}. \\

We will assume familarity with the basics of stability theory for which \cite{pil-book} is a reference. Definability means possibly over  parameters and $A$-definable means definable with parameters from $A$. It will be convenient, especially in so far as the results assuming properties of definable Galois groups are concerned, to place further restrictions on the theory $T$. \\
\newline
{\bf ASSUMPTION.} $T$ is a complete stable theory, $M$ a saturated model, and there is a fixed collection  ${\mathcal D}$ of strongly minimal sets defined over $\emptyset$ such that any type (over any set of parameters) is nonorthogonal to some $D\in {\mathcal D}$. \\

Under this assumption one knows that all definable sets in $M^{eq}$ have finite Morley rank ( and in fact Morley rank is definable and equals $U$-rank). 
Let ${\mathcal D}_{nm}$ denote the set of nonmodular strongly minimal sets in $\mathcal D$. Let${\mathcal D}^{eq}$ be the elements in $M^{eq}$ in the definable closure of tuples taken from various $D\in {\mathcal D}$. Likewise for $({\mathcal D}_{nm})^{eq}$. 
\\

Generally we work in $M^{eq}$, and $a,b,c,..$ range over such elements. $A, B, ..$ range over small (usually finite) subsets of $M^{eq}$. As usual we feel free to identify formulas and definable sets.
\begin{defi} (i) Let $X$ be a definable set. We say that $X$ is {\em internal} to $\mathcal D$ if there is a definable set $Y$ in ${\mathcal D}^{eq}$ and a definable bijection between $X$ and $Y$.
\newline
(ii) We say that $X$ is almost internal to $\mathcal D$ if there is definable $Y$ in ${\mathcal D}^{eq}$ and a surjective definable function $f$ from $X$ to $Y$ with finite fibres.
\newline
(iii) We say that the stationary type $p(x) = tp(a/A)$ is (almost) internal to ${\mathcal D}$ if some formula $\phi(x)\in p(x)$ is.
\newline
(iv) Likewise for $({\mathcal D}_{nm})^{eq}$. 
\end{defi}

The important thing in part (i) of the definition is that $X$ could be $\emptyset$-definable, and internal to $\mathcal D$ but any $f$ witnessing it needs additional parameters for its definition. \\

\begin{remark} Definition 1.4 (ii) agrees with the ``usual" definition of a stationary type $p(x)\in S(A)$ being (almost) internal to $\mathcal D$: namely that for some $B\supseteq A$, and $a$ realizing the nonforking extension of $p$ over $B$, there is a tuple $c$ of elements from various $D\in \mathcal D$ such that $a\in  dcl(B,c)$ ($a\in acl(B,c)$). 
\end{remark}

The following says that any type is {\em analyzable} in $\mathcal D$. 
\begin{fact} Work over any given set of parameters, algebraically closed if you want. For any $a\notin acl(\emptyset)$ there are $a_{0},..,a_{n}\in dcl(a)$ such that $stp(a_{0})$ is nonalgebraic and internal to $\mathcal D$, $tp(a_{i}/a_{0}...a_{i-1})$ is (nonalgebraic, stationary, and)  internal to $\mathcal D$ for all $i=1,..,n$ and $a\in acl(a_{0},..,a_{n})$ (in fact $tp(a/a_{0}...a_{n-1})$ is almost internal to $\mathcal D$).
\end{fact} 

In particular one has:
\begin{fact} For any $a\notin acl(\emptyset)$ there is $b\in dcl(a)$ with $RM(tp(b)) < RM(tp(a))$, and $a'\in dcl(a)$ such that $tp(a'/b)$ is internal to $\mathcal D$ and $a\in acl(b,a')$.
\end{fact} 

\begin{remark} If $tp(a/A)$ is almost internal to $\mathcal D$ then there is $b\in dcl(aA)$ such that $tp(b/A)$ is internal to $\mathcal D$ and $a\in  acl(b,A)$. 
\end{remark}

We repeat the definition of the canonical base property in the current framework:
\begin{defi} $T$ has the $CBP$ if (working over any set of parameters), if $tp(a/B)$ is stationary and $b = Cb(tp(a/b))$ then $stp(b/a)$ is almost internal to $\mathcal D$.
\end{defi}

As remarked earlier this is equivalent to requiring that $stp(b/a)$ is almost internal to ${\mathcal D}_{nm}$ (because by Chatidakis' general results 
$stp(b/a)$ is analysable in ${\mathcal D}_{nm}$).\\

Now we recall {\em definable Galois groups} (sometimes called binding groups or liason groups). The theory has origins in works of Zilber, with 
further inputs from Hrushovski and Poizat. A general account appears in Chapter 7 of \cite{pil-book}. 
\begin{fact}  Suppose $p(x)\in S(A)$ is internal to $\mathcal D$. Then there is an $A$-definable group $G$ and an $A$-definable action of $G$ on the set $Y$ of realizations of $p(x)$ which is naturally isomorphic (as a group action) to the group of permutations of $Y$ induced by automorphisms of $M$ which fix pointwise $a$ and all $D\in {\mathcal D}$. Moreover $G$ is itself internal to $\mathcal D$.
\end{fact}

We will call any such group $G$ as in Fact 1.10 a {\em definable Galois group} in $T$.

\begin{defi} Let $G$ be a definable group, defined over $A$ say. We say that $G$ is {\em rigid} if every definable, connected subgroup of $G$ is defined over $acl(A)$.
\end{defi} 

\vspace{2mm}
\noindent
The following comments are either obvious or left as exercises.
\begin{remark} (i) Any rigid group (of finite Morley rank) must by nilpotent-by-finite.
\newline
(ii) $1$-based groups are rigid.
\newline
(iii) Strongly minimal groups are rigid.
\newline
(iv) In $ACF_{0}$ a connected definable group (i.e. connected algebraic group) is rigid iff it is an extension of an abelian variety by a product of the additive group with an algebraic torus.  
\end{remark}





Finally a word on attributions: The first author circulated a one and a half page note (in late 2011) outlining a counterexample to the $CBP$ and also remarking among other things that rigidity of definable Galois groups should imply the $CBP$. The second and third author checked the details of the counterexample, and also noted that the same configuration witnessing failure of the $CBP$ also witnesses the failure of the ``group version" Fact 1.3. This appears in Section 3 of the current paper which is a mild simplification of the original example. They also found a proof of the $CBP$ (in fact a strong version) under the rigidity hypothesis. This, together with a certain ``local" consequence of rigidity of a Galois group, as well as a suitable generalization to arbitrary stable theories, appears in Section 2.

\section{Rigidity of definable Galois groups implies the $CBP$}

We work under the ASSUMPTION from section 1, and prove:

\begin{thm} Suppose all definable Galois groups in $T$ are rigid. Then $T$ has the canonical base property: in fact in the strong form that if $b = Cb(stp(a/b))$ then $b\in acl(a,{\mathcal D})$

\end{thm}
\begin{proof} We will use freely elementary closure properties of (almost) internality, as well as basic facts that if some nonforking extension of a (stationary) type $p(x)$ is (almost) internal to $\mathcal D$ then so is $p$. The proof is remarkably similar to that of the first authors ``socle theorem" in \cite{udi96} proving the ``group version" of the $CBP$ (Fact 1.3) under certain rigidity assumptions on the definable group $G$.\\

Work over any base set. Sometimes we use $\mathcal D$ to mean the union of the sets of realizations of all the formulas in $\mathcal D$, an $Aut(M)$-invariant set. Let $c,a$ be such that $tp(c/a)$ is stationary and $a = Cb(tp(c/a))$. We will prove by induction on 
$U(tp(c))$ (= $RM(tp(c))$) that $tp(a/c)$ is almost internal to $\mathcal D$.\\

Applying Fact 1.7 to $stp(c)$ there is $d\in dcl(c)$ such that $0\leq U(tp(d)) < U(tp(c))$ and $tp(c/d)$ is almost internal to $\mathcal D$. Note that $tp(d/a)$ is stationary. Let $b\in dcl(a)$ be $Cb(tp(d/a))$.  By induction hypothesis we have:
\newline
{\em Claim I.}  $b\in acl(d,{\mathcal D})$ so also $b\in acl(c, {\mathcal D})$. \\

\noindent
{\em Claim II.}  $stp(a/b)$ is almost internal to $\mathcal D$. 
\newline
{\em Proof.} Let $(c_{1}d_{1}, c_{2}d_{2},....)$ be a Morley sequence in $tp(c,d/a)$. So $a\in dcl(c_{1},d_{1},..,c_{n},d_{n})$ for some $n$. 
 Now
$stp(c_{1},..,c_{n}/d_{1}....d_{n})$ is almost internal to $\mathcal D$, hence $stp(a/d_{1},..,d_{n})$ is almost internal to $\mathcal D$, so also $stp(a/d_{1},..,d_{n},b)$ is almost internal to $\mathcal D$.

But $(d_{1}, d_{2},...)$ is a Morley sequence in $tp(d/a)$, hence $a$ is independent from $d_{1},...,d_{n}b$ over $b$, so 
 $stp(a/b)$ is almost internal to $\mathcal D$, as required. \\
 
 There is no harm in adding something in $acl(b)\cap dcl(ab)$ to $b$ so we will assume $tp(a/b)$ is stationary. By Remark 1.8 let $a'\in dcl(a)$ be such that $tp(a'/b)$ is internal to $\mathcal D$, $b\in dcl(a')$, and $a\in acl(a')$.
  Let $G_{b}$ be the definable Galois group corresponding to $tp(a'/b)$ given by Fact 1.10. Namely we have a $b$-definable action of $G_{b}$ on the set $Y$ of realizations of $tp(a'/b)$ isomorphic to the action of $Aut(M/b,{\mathcal D})$ (quotiented by the pointwise fixator of $Y$) on $Y$.
  Of course $tp(a'/cb)$ is also internal to $\mathcal D$, with corresponding Galois group $L$ say. The functoriality of the Galois actions gives $L$  as a subgroup of $G_{b}$ and its action on the set of realizations of $tp(a'/cb)$ by automorphisms which fix pointwise $bc$ and all sets in $\mathcal D$ is induced by the action of $G_{b}$ on $Y$ by automorphisms which fix pointwise $b$ and the sets in $\mathcal D$, so this action is also $b$-definable. 
Let $H$ be the connected component of $L$, so by the rigidity assumption $H$ is defined over $acl(b)$. There is no harm in assuming it to be defined over $b$ and we write $H$ as $H_{b}$. Note that the orbit of $a'$ under $L$ is precisely the set of realizations of $tp(a'/bc,{\mathcal D})$. Hence the orbit of $a'$ under $L$ is definable over $b,c, {\mathcal D}$. This orbit breaks into finitely many orbits under $H_{b}$. Hence if $Z$ denotes
the orbit of $a'$ under $H_{b}$ we have:
\newline 
{\em Claim III.} $Z$ is definable over $acl(bc{\mathcal D}$). \\

Let $e$ be a canonical parameter for $Z$. So note $e\in dcl(ab) = dcl(a)$, and by Claim III, $e\in acl(bc,{\mathcal D})$. \\

\noindent
{\em Claim IV.}  $tp(c/e)$ implies $tp(c/a')$.
\newline
{\em Proof.}  Note that $tp(h\cdot a'/b,c,{\mathcal D}) = tp(a'/b,c,{\mathcal D})$ for any $h\in H_{b}$. 
So if $\phi(x,y)$ is a formula such that $M\models \phi(c,a')$ then $M\models \phi(c,a'')$ for all $a''\in Z$. Now the formula $\psi(x)$ which says that
$\forall y\in Z(\phi(x,y))$ is over $e$ so we write it as $\psi(x,e)$. And clearly $M\models \psi(x,e) \rightarrow \phi(x,a')$. So Claim IV is established.\\

As $a\in acl(a')$, $tp(c/a)$ does not fork over $a'$. By Claim IV, $tp(c/a)$ does not fork over $e$. Hence as $a = Cb(tp(c/a))$, $a\in acl(e)$. Together with Claim III, we obtain:\\

\noindent
{\em Claim V.} $a\in acl(b,c, {\mathcal D})$.\\

\noindent
By Claim I, we conclude that $a\in acl(c, {\mathcal D})$, as required. 
\end{proof}

\begin{remark} It is not hard to find examples (say in differentially closed fields) where the $CBP$ is true but the conclusion of Theorem 2.1 fails.

\end{remark}

We finish this section with a response to a question of Rahim Moosa as to what the ``local" content of the proof above is. Namely what are the consequences, regarding canonical bases and internality, of the rigidity of a given definable Galois group.  We will give an answer and then give a suitable generalization of Theorem 2.1 to arbitrary stable theories. 
So we consider now a general stable theory $T$. Again we work in a saturated model $M$ of $T$. Let $Q$ denote some $Aut(M)$ invariant subset of $M^{eq}$ (such as the set of realizations of a partial type over $\emptyset$).  Remark 1.5 gives the definition of the (almost) internality of a stationary type $p(x) \in S(A)$ to $Q$. If $p(x)\in S(A)$ is internal to $Q$ then as in Fact 1.10 there is a type-definable over $A$ group $G$ and $A$-definable action of $G$ on the set of realizations $Y$ of $p$, isomorphic (functorially) to the group of permutations of $Y$ induced by automorphisms of $M$ which fix $A$ and $Q$ pointwise. We will say that $G$ is rigid if any connected type-definable subgroup of $G$ is type-definable over $acl(A)$.
With this notation we have the following:\\

\begin{lem} ($T$ stable.) Suppose $tp(a)$ is internal to $Q$, and suppose moreover that the type-definable Galois group $G$ is rigid. Then for any $c$, if $b =  Cb(stp(c/a))$, then $b\in acl(c,Q)$. 
\end{lem}
\begin{proof} We give a sketch proof as it really is a rewriting of Claims III, IV, V, in the proof of Theorem 2.1. Note that $b$ is now a possibly infinite tuple. Let $L$ be the Galois group corresponding to $tp(a/acl(c))$ which is still internal to $Q$. Then the orbit of $a$ under $L$ is type-definable over $(acl(c),Q)$. 
Let $H$ be the connected component of $L$. Then $H$ is type-definable over $acl(\emptyset)$, and the orbit of $a$ under $H$ is definable over $acl(c,Q)$.
Now the orbit of $a$ under $H$ has also a ``canonical parameter" $e$ which may be an infinite sequence. Hence $e\in acl(c,Q)$. Note that 
$e\in acl(a)$. As in the proof of Claim IV above, $tp(c/e)$ implies $tp(c/a)$. Hence $tp(c/acl(a))$ does not fork over $e$, and so $b = Cb(stp(c/a))$ is in $acl(e)$.  Hence  $b\in acl(c,Q)$ as required.
\end{proof}

The Lemma above also holds when $a$ is a possibly infinite tuple (and so $G$ is a so-called $*$-definable group). \\

Now assume $tp(a)$ is stationary. We say that $tp(a)$ is {\em analysable} in $Q$ if there is a sequence $a_{\alpha}$ $\alpha \leq \beta$ such that
$stp(a_{\alpha}/\{a_{\gamma}:\gamma < \alpha\})$ is internal to $Q$ for all $\alpha \leq\beta$, and $a\in acl(a_{\alpha}:\alpha\leq \beta)$.   In fact 
the $a_{\alpha}$ may be infinite tuples, and one can assume they are in $dcl(a)$. We define $\ell_{1}^{Q}(a)$ to be the maximal (possibly infinite) tuple $b$ in $acl(a)$, such that $stp(b)$ is internal to $Q$. Then an adaptation of Theorem 3.4(2) of \cite{pal-wag} yields:
\begin{fact} ($T$ stable.) Assume $tp(a)$ is analysable in $Q$. Then  $\ell_{1}^{Q}(a)$ {\em dominates} $a$ over $\emptyset$: whenever $c$ is independent from $\ell_{1}^{Q}(a)$ over $\emptyset$, then $c$ is independent from $a$ over $\emptyset$. 
\end{fact}

\begin{thm} ($T$ stable.)  Suppose that for any $b,A$ such that $stp(b/A)$ is internal to $Q$, the corresponding  Galois group is rigid. Then, whenever $tp(a)$ is analysable in $Q$ we have: for any $c$, if $b$ = $Cb(stp(c/a)$, then $b\in acl(c,Q)$. 
\end{thm} 
\begin{proof} So assume $tp(a)$ is analysable in $Q$, and let $c$ be arbitrary. Let $A = acl(c,Q)\cap acl(a)$, and note that $tp(a/A)$ is still analysable in $Q$. Let $a' = \ell_{1}(a/A)$.  By Lemma 2.3 (working over $A$), and our assumption on the Galois groups, $Cb(stp(c/A,a')) \subseteq
acl(c,A,Q)\cap(acl(A,a') \subseteq acl(c,Q)\cap acl(a) = A$. So $c$ is independent from $a'$ over $A$. Hence by Fact 2.4 (working over $A$), $c$ is independent from $a$ over $A$. Hence $Cb(stp(c/a)) \subseteq A \subseteq acl(c,Q)$, as required.

\end{proof}

\section{The Counterexample}

In this section we will construct an $\aleph_1$-categorical structure of finite rank where the $CBP$ does not hold. In fact the same data will witness failure of the $CBP$ in two different ways, the second way being the failure of Fact 1.3

Before to describe the structure, we shall recall the definition of ``tangent bundle" $TV$ of an irreducible affine algebraic variety $V$ (although unless $V$ is smooth, $TV$ will not actually be a vector bundle). Let $K$ be an algebraically closed field.
\begin{defi} Assume that $V\subseteq K^n$ is an irreducible affine algebraic variety whose ideal over $K$ is generated by $P_1,\ldots,P_m$. The {\em tangent bundle $TV$ of $V$} is the affine algebraic variety contained in $k^{2n}$ and defined by equations $$P_j(\bar x)=0\quad\mbox{ and }\quad\sum_{i=1}^n \frac{\partial P_j}{\partial x_i}(\bar x)u_i=0\quad\mbox{ for $j=1,\ldots,m$}.$$ \end{defi} If $\pi$ is the projection from $TV$ on $V$, the tangent space of $V$ at $a\in V$ is the vector space $T_a V=\pi^{-1}(a)$. Note that, if $V=K^n$ then $TV=K^{2n}$.

\subsection{The structure and its properties} Our counterexample $M$ will be essentially a reduct of an algebraically closed field of characteristic $0$.

So let $K$ be a saturated algebraically closed field of characteristic $0$, so the field of complex numbers if one wishes. The universe of $M$ will be affine $2$-space over $K$, namely $K\times K$. The basic relations and functions on $M$ will consist of:
\newline
(i) a predicate $P$ say for the first copy of $K$, namely  $\{(a,0):a\in K\}$,
\newline
(ii) The full field structure $+,-,\cdot, 0, 1$ on $P$,
\newline
(iii) The projection $\pi$ from $K\times K$ to $P$  (i.e. $\pi(a,b) = (a,0)$).
\newline
(iv) The action $*$ of $P$ on $K\times K$, given by  $a * (b,c) = (b,c+a)$,
\newline
(v)  For each irreducible subvariety $W$ of $K^{n} = P^{n}$ defined over ${\mathbb Q}^{alg}$, a predicate $P_{W}$ for $TW \subset (K\times K)^{n}$. 

\vspace{2mm}
\noindent
It will be convenient to consider $M$ as a two sorted structure $(P,S)$. The sort $P$ is just an algebraically closed field, and we view $S$ as the tangent bundle of $P$ equipped with relations and functions above. $M$ is obviously a reduct of the structure  $(P,P\times P)$ (with the identification of the second sort with the square of the first), so interpretable in $(K,+,\cdot)$. Moreover the induced structure on $P$ is just the field structure. The structure $M$ is clearly saturated.  Note that $S$ is analysable in $2$ steps in $P$: for any $a\in S$, $\pi(a)\in P$ and the fibre $S_{\pi(a)}$ in which $a$ lives is in bijection with $P$ definably over any element in the fibre. Hence $Th(M)$ is $\aleph_{1}$-categorical. \\

We now aim to show that $S$ is not almost internal to $P$. It will be done by observing that $M$ has many automorphisms, acting trivially on $P$, induced by {\em derivations} of $K$. 

\begin{lem} Let $\partial$ be any derivation of $K$. For $(a,b)\in S$,
$\sigma_{\partial}(a,b) = (a, b+\partial(a))$ (so in particular $\sigma$ fixes $P$ pointwise).  Then $\sigma_{\partial}$ is an automorphism of $M$
\end{lem}
\begin{proof}  Clearly $\sigma$ is a permutation of $M$ and preserves the the relations and functions in (i), (ii), (iii), (iv) above.
So it just remains to show that $\sigma(P_{W}) = P_{W}$ for any irreducible variety $W\subset P^{n}$ over ${\mathbb Q}^{alg}$. 
 Suppose that $(\bar a,\bar u)$ (strictly speaking  $((a_{1},u_{1}),..,(a_{n},u_{n})$) is in $S^{n}$.  We will show that  
\newline
(*) $(\bar a,\bar u) \in P_{W}$
if and only if $\sigma_{\partial}(\bar a,\bar u)$ is in $P_{W}$.  
\newline 
We may assume ${\bar a}\in W$. 
Let $Q(x_{1},..,x_{n})$ be a polynomial over ${\mathbb Q}^{alg}$ which is in the ideal of $W$.  So $Q(\bar a) = 0$. So applying $\partial$ to both sides and noting that it vanishes on the coefficients of $Q$ we see that

$$\sum_{i=1}^n \frac{\partial Q}{\partial x_i}(\bar a)\partial(a_i)=0.$$

Hence

$$\sum_{i=1}^n\frac{\partial Q}{\partial x_i}(\bar a)u_i=0$$ if and only if

$$\sum_{i=1}^n\frac{\partial Q}{\partial x_i}(\bar a)(u_i+\partial(a_i))=0.$$

This proves (*) and hence the Lemma.
\end{proof}

\begin{cor} Let $a\in S$ and $B\subset S$ be such that  $\pi(a)\notin acl(\pi(B))$  (in the structure $P$). Then $a\notin acl(B\cup P)$. 
\end{cor}
\begin{proof} Let $\partial$ be a derivation of $K$ which vanishes on $\pi(B)$ but not on $\pi(a)$, Suppose $\partial(\pi(a)) = c \neq 0$.
Then $\sigma_{\partial}$ fixes $B$ and $P$ pointwise but has infinite orbit on $a$. Namely writing $a = (\pi(a),d)$, the orbit is  $\{(\pi(a),d+ nc):n=1,2,...\}$.  This proves the Corollary.
\end{proof} 

The following lemma summarizes some facts about definability, rank, etc. in $M$ which will be useful. 

\begin{lem} (i) Let $a\in  P$. Then $\pi^{-1}(a)$ is strongly minimal.
\newline
(ii) If moreover $a\notin acl(\emptyset)$ then for any $b\in \pi^{-1}(a)$, $RM(tp(b/a)) = 1$, and in fact $tp(b/a)$ implies $tp(b/P)$. 
\newline
(iii) Suppose $b_{1},..,b_{n}\in S$, $a_{i} = \pi(b_{i})$ for $i=1,..,n$ and the $a_{i}\in P$ are  independent (generic) in $P$. Then 
$\{b_{1}, b_{2},..,b_{n}\}$ is independent (so $RM(tp(b_{1},..,b_{n})) = 2n$ by (ii) and additivity of rank), $\{b_{1},..,b_{n}\}$ is independent over $(a_{1},..,a_{n})$ and in fact $tp(b_{1},..,b_{n}/a_{1},..,a_{n})$ implies $tp(b_{1},..,b_{n}/P)$
\newline
(iv)  Suppose $b\in S$, $B\subset P\cup S$ and $\pi(b)\notin acl(B)$. Then $b$ is independent from $B$ over $\emptyset$.
\newline
(v)  Let $b\in S$ be ``generic"  (i.e. $RM(tp(b)) = 2$, equivalently by (ii), $RM(tp(\pi(a)) = 1$). Then $tp(a)$ is not almost internal to $P$.
\end{lem}
\begin{proof} (i) is clear because the fibre is definably isomorphic to $P$ under the action $*$ (although the definable isomorphism needs a parameter).
\newline 
(ii)  By Corollary 3.3,  $b\notin acl(P)$ so by (i) $tp(b/a)$ and $tp(b/P)$ have Morley rank $1$.
\newline
(iii) We will prove it for the case where $n=2$ (an inductive proof yields the general statement). First by (ii), 
\newline
(*) $b_{1}\notin acl(P)$. 
\newline
As $a_{2}\notin acl(a_{1}$, by Corollary 3.3, 
\newline (**) $b_{2}\notin acl(P,b_{3})$. 
The conclusion (using (i)) is that $\{b_{1},b_{2}\}$ is independent over $(a_{1},a_{2})$, in particular $RM(tp(b_{1},b_{2}/a_{1}a_{2})) = 2$. Additivity of Morley rank implies $RM(tp(b_{1},b_{2})) = 4$ and $b$ is independent from $b_{2}$ over $\emptyset$. But we also conclude from (*) and (**) that $b_{1}$ and $b_{2}$ are independent over $P$. This yields (iii).
\newline
(iv) Let $a=\pi(b)$. By Corollary 3.3, $b\notin acl(a,B)$. As $a\notin acl(B)$, we get that $RM(tp(b/B)) = 2$. So as $RM(tp(b)) = 2$ (by (iii)), $b$ 
is independent from $B$ over $\emptyset$. 
\newline
(v) Suppose $B\subset S$ and $b$ is independent from $B$ over $\emptyset$. Then $a = \pi(b) \notin acl(B)$, so by Corollary 3.3, 
$b\notin acl(B\cup P)$. 


\end{proof}

\begin{question}  Note that we can do exactly the same construction of $M$ in positive characteristic, and we ask whether the $CBP$ holds. We believe yes. In fact it is conceivable that any theory interpretable in $ACF_{p}$ has the $CBP$.

\end{question}

\subsection{The example}  We describe the configuration (essentially as outlined by the first author) which directly violates the canonical base property.  Let $a, b, c$ be generic independent in $P$, and let $d = (1-ac)/b$. Then $RM(tp(a,b,c,d)) = 3$ and $(a,b,c,d)$ is a generic point of the irreducible smooth algebraic subvariety $W$ of  $P^{4}$ defined by 

$$xw + yz = 1$$. 

It is well-known that $(a,b,c,d)$ witnesses the nonmodularity of $P$: namely
\begin{fact}  $(a,b) = Cb(tp(c,d/a,b))$ and  $acl(a,b)\cap acl(c,d) = \acl(\emptyset)$  (either working in $ACF_{0}$ or in $M$). 
Moreover, if $(c_{1},d_{1})$ realizes $tp(c,d/a,b)$ independently of $(c,d)$ over $(a,b)$, then $(c,d)$ is independent from $(c_{1},d_{1})$ over $\emptyset$, and $(a,b)\in acl(c,d,c_{1},d_{1})$. 
\end{fact}

We consider now $P_{W} \subset S_{4}$. We will consider again points of $S$ (externally) as pairs $(a,u)$, although the structure $M$ only sees $(a,u)$ as a single point of $S$ whose projection to $P$ is $a$.  
The fibre $P_{W}(a,b,c,d)$ of $P_{W}$ over $(a,b,c,d)$ is of course definable in $M$, and (externally) is the set of $(u,v,r,s)$ such that 

$$ cu  + vd + ra + bs  = 0$$. 

Let us fix such $u, v, r, s$. By Lemma 3.4(iii)  $\{(a,u), (b,v), (c,r)\}$ is independent and the Morley rank of its type is $6$.  The equation above tells us
that there is a unique $s$ such that  $(u,v,r,s)\in P_{W}(a,b,c,d)$. Hence $RM(tp((a,u), (b,v), (c,r), (d,s)) = 6$ and in fact $(a,u), (b,v), (c,r), (d,s)$ is a ``generic point" of $P(W)$ as well as a generic point of $P(W)(a,b,c,d)$ over $(a,b,c,d)$. With this notation we will prove:

\begin{thm} $((a,u),(b,v))$ is interalgebraic with the canonical base of $stp(((c,r),(d,s))/(a,u), (b,v))$ but  $tp(((a,u),(b,v)/(c,r),(d,s))$ is not almost internal to $P$
\end{thm}

As $P$ is the unique strongly minimal set up to  nonorthogonality (in fact any strongly minimal set in $M$ is in definable bijection with a strongly minimal set definable on the $P$-sort), Theorem 3.6 shows that $Th(M)$ does not have the $CBP$

\begin{proof} (of Theorem 3.6) 
We start by proving that $((a,u),(b,v))$ is interalgebraic with the canonical base of $stp(((c,r),(d,s))/(a,u), (b,v))$. 
Let $((c_{1},r_{1}), (d_{1}, s_{1}))$  realize $stp(((c,r),(d,s))/(a,u), (b,v))$ such that  
$((c_{1},r_{1}, (d_{1}, s_{1}))$ is independent from $((c,r),(d,s))$ over $\{(a,u), (b,v)\}$. \\

\noindent
{\em Claim I.} $((c_{1},r_{1}), (d_{1}, s_{1}))$ is independent from 
$((c,r),(d,s))$ over $\emptyset$  (in $M$). 
\newline 
{\em Proof.} By Lemma 3.4(iii)  $(c,d)$ is independent from $((a,u),(b,v))$ over $(a,b)$. Hence $(c_{1},d_{1})$ is a realization of $tp(c,d/a,b)$ independent from $(c,d)$ over $(a,b)$. By Fact 3.5  $(c,d)$ is independent from $(c_{1},d_{1})$ over $\emptyset$ (in $P$). Hence $c,d,c_{1},d_{1}$ are independent generic in $P$, and we finish by Lemma 3.4(iii).\\

\noindent
{\em Claim II.} $((a,u), (b,v)) \in acl((c,r), (d,s), (c_{1},r_{1}), (d_{1},s_{1}))$.
\newline
{\em Proof.} By Fact 3.5 , $(a,b) \in acl(c,d,c_{1},d_{1})$. So we can add $a,b$ to the right hand side in the statement of the claim.
\newline
We have:
$$ cu  + dv + ar + bs  = 0$$

and

$$ c_{1}u  + d_{1}v + ar_{1} + bs_{1}  = 0$$.

Let $X$ be the set of $(u',v')$ in $\pi^{-1}(a,b)$ such that $cu'  + dv' + ar + bs  = 0$, and $Y$ the set of $(u',v')$ in $\pi^{-1}(a,b)$ such that $c_{1}u'  + d_{1}v' + ar_{1} + bs_{1}  = 0$. Then $X$ and $Y$ are definable (over $a,b,(c,r), (d,s), (c_{1},r_{1}), (d_{1},s_{1})$) in $M$, $X$ is (seen externally) a translate of the vector subspace $V_{1}$ of $\pi^{-1}(a,b)$ defined by  $cu' + dv' = 0$. Likewise $Y$ is a translate of the vector subspace $V_{2}$ defined by $c_{1}u' + d_{1}v' = 0$. As $c,d,c_{1}, d_{1}$ are independent these are distinct one-dimensional subspaces. Hence $X$ and $Y$ intersect in a point, namely $(u,v)$ and the proof of Claim II is complete.\\

From Claims I and II we conclude on general grounds that $((a,u),(b,v))$ is interalgebraic with the canonical base of $stp(((c,r),(d,s))/(a,u), (b,v))$.

On the other hand by the independence of $\{a,c,d\}$, we have by Lemma 3.4(iii) that $(a,u)$ is independent from $((c,r),(d,s))$ over $\emptyset$. So by Lemma 3.4(v), 
$tp((a,u)/(c,r), (d,s))$ is not almost internal to $P$, so also
\newline
$tp((a,u),(b,v))/(c,r), (d,s))$ is not almost internal to $P$. This completes the proof of Theorem 3.6.

\end{proof}

\subsection{Group version} Here we point out that the same data as in subsection 3.2 yields a definable group and a type with trivial stabilizer which is not almost internal to $P$, so by Fact 1.3 gives another witness to the failure of the $CBP$.  In  fact the proof is easier than that of Theorem 3.6. \\

Before stating and proving the result, let us note that the natural additive group structure on $S$ is $\emptyset$-definable in $M$, because the graph of addition on $S$ is precisely the tangent bundle of the graph of addition on $K = P$. So we can and will speak of $(S,+)$ and also $(S^{n},+)$ etc.
Let $((a,u), (b,v), (c,r), (d,s))\in S^{4}$ be as in subsection 3.2. Namely a ``generic point" of $P_{W}$. Let $q$ be the (strong) type of
$((a,u), (b,v), (c,r), (d,s))$ over $\emptyset$. 

\begin{thm} Working in the $\emptyset$-definable group $(S^{4},+)$, $Stab(q)$ is trivial, but $q$ is not almost internal to $P$.
\end{thm}
\begin{proof} We already know that $tp((a,u))$ is not almost internal to $P$, so the same thing is true of $q$.\\

For the rest, we first show that $Stab(tp(a,b,c,d))$ in $(P^{4},+)$ is trivial. Let $(g_{1},g_{2},g_{3},g_{4})\in P^{4}$ be independent from $(a,b,c,d)$ over $\emptyset$ and suppose that
$(g_{1},g_{2},g_{3},g_{4}) + (a,b,c,d)$ realizes $q$. So

$$(g_1+a)(g_3+c) +(g_2+b)(g_4+d)= 1 =  ac+bd,$$

We deduce that

$$g_1g_3 + g_1c+ ag_3+ g_2g_4+bg_4+ g_2d=0.$$

This clearly contradicts the independence assumption, unless $g_{i} = 0$ for $i=1,2,3,4$.\\

Now we show $Stab(q)$ is trivial. By what we have just seen any element of $Stab(q)$ must be of the form $((0,x), (0,y), (0,w), (0,z))$. We assume that this element is independent with $((a,u), (b,v), (c,r), (d,s))$ over $\emptyset$ and that 
$((0,x), (0,y), (0,w), (0,z)) + ((a,u), (b,v), (c,r), (d,s))$ realizes $q$, and try to get a contradiction, unless $x=y=w=z=0$.
 So clearly
 
 $$c(x+u)+d(y+v)+ a(w+r) + b(z+s)=0.$$
 
 whence
 
 $$cx + dy + aw+ bz = 0$$
 
 This however is not expressed by a formula in $M$. So we let 
\newline 
$((0,x'), (0,y'), (0,w'), (0,z'))$ realize the same strong type as $((0,x), (0,y), (0,w), (0,z))$ and independent with $((0,x), (0,y), (0,w), (0,z)), ((a,u), (b,v), (c,r), (d,s)))$ over $\emptyset$. So we also have:
 
 $$cx' + dy' + aw'+ bz' = 0$$
 
 Now let $x''\in P$ be such that  $x''*(0,x) = (0,x')$  (i.e. externally $x'' = x' - x$). Likewise for $y'',w'',z''$. So we see that
 
 $$cx'' + dy'' + aw''+ bz'' = 0$$
 
 But $(x'',y'',w'',z'')$ is independent from $a,b,c,d)$  (in $P$) and we deduce easily that  $x''=y'' = w'' = z'' = 0$. Hence $x=x',y=y',w=w',z=z')$. 
 The independence of  $((0,x), (0,y), (0,w), (0,z))$ and $((0,x'), (0,y'), (0,w'), (0,z'))$ over $\emptyset$ implies that 
\newline
$((0,x), (0,y), (0,w), (0,z))\in acl(\emptyset)$. Hence $Stab(q)$ is finite, so trivial.
\end{proof}

\bibliographystyle{plain}

\end{document}